\documentclass{article}
\usepackage{amsmath}
\usepackage{amsthm}
\usepackage{,amssymb}
\usepackage[dvips]{graphicx,color}

\theoremstyle{plain}
\newtheorem{theorem}{Theorem}[section]
\newtheorem{proposition}[theorem]{Proposition}
\newtheorem{corollary}[theorem]{Corollary}
\newtheorem{lemma}[theorem]{Lemma}
\newtheorem{example}[theorem]{Example}

\title{Operator monotonicity of some functions}

\author{Masaru Nagisa
          \thanks{Department of Mathematics and Informatics, 
                     Graduate School of Science, Chiba University, Chiba, 263-8522, Japan. 
                      E-mail: \texttt{nagisa@math.s.chiba-u.ac.jp}} 
   \and Shuhei Wada
          \thanks{Department of Information and Computer Engineering, 
                    Kisarazu National College of Technology, 
                    Kisarazu City, Chiba, 292-0041, Japan. 
                    E-mail: \texttt{wada@j.kisarazu.ac.jp}}
}
\date{}

\begin{document}

\maketitle

\begin{abstract}
We investigate the operator monotonicity of the following functions:
\[ f(t) = t^\gamma \frac{(t^{\alpha_1} -1)(t^{\alpha_2} -1)\cdots (t^{\alpha_n} -1)}
       {(t^{\beta_1} -1)(t^{\beta_2} -1)\cdots (t^{\beta_n} -1)} \qquad (t\in (0,\infty) ), \]
where $\gamma \in \mathbb{R}$ and $\alpha_i , \beta_j >0$ with
$\alpha_i \neq \beta_j$ ($i,j=1,2,\ldots,n$).
This property for these functions has been considered by V.E. Szab\'o \cite{szabo}.
\end{abstract}


\section{Introduction and Main results}
We consider the following functions on $(0,\infty)$:
\[ f(t) =\begin{cases}
    t^\gamma \prod_{i=1}^n (t^{\alpha_i} -1)/(t^{\beta_i} -1) \qquad & \text{ if } t\neq 1\\
    \prod_{i=1}^n \alpha_i/\beta_i  & \text{ if } t=1 \end{cases} , \]
where $\gamma \in \mathbb{R}$ and $\alpha_i , \beta_i >0$ with
$\alpha_i \neq \beta_j$ ($i,j=1,2,\ldots,n$).
These functions has been treated by V.E.S. Szab\'o and he has discussed 
their operator monotonicity in \cite{szabo}.
He gave a sufficient  condition for that $f(t)$ becomes operator monotone on 
$(0,\infty)$.
But his argument contained some errors and it is known the existence of 
a function $f(t)$ which satisfies this condition and is not operator monotone.
We will  make a new sufficient condition for that $f(t)$ becomes 
operator monotone on $(0,\infty)$.

Let $g(t)$ be a real valued continuous function on $(0,\infty)$.
For a positive, invertible bounded linear operator $A$ on a Hilbert space $\mathcal{H}$,
we denote by $g(A)$ the continuous functional calculus of $A$ by $g$.
We call $g$ operator monotone if $0< A \le B$ implies $g(A)\le g(B)$.

We assume that $g$ is not constant.
We call $g$ a Pick function on $(0,\infty)$ if $g(t)$ has an analytic continuation $g(z)$
to the upper half plane $\mathbb{H}_+ =\{z\in \mathbb{C} \mid {\Im} z>0 \}$ and
$g(z)$ maps  $\mathbb{H}_+$ into $\mathbb{H}_+$, where $\Im z$ means the imaginary part of $z$.  
It is known that a Pick function is operator monotone and conversely 
a non-constant operator monotone function is a Pick function (L\"owner's theorem),
see \cite{bhatia1}, \cite{donoghue}, \cite{hiai} and \cite{hiaipetz}.
In this paper, we will show that some $f(t)$ is operator monotone by proving that
$f(t)$ is a Pick function.
For a continuous function $g$ on $(0,\infty)$,
$g$ becomes operator monotone if there exists a sequence $\{g_n\}$ 
of operator monotone functions such that $\{g_n\}$ pointwise converges to 
$g$ on $(0,\infty)$.

We assume that $f(t)$ is a Pick function.
By definition, there exists a holomorphic function $f(z)$ on $\mathbb{H}_+$ 
with $f(\mathbb{H}_+)\subset \mathbb{H}_+$,
\begin{gather*}  f(z) = z^\gamma \prod_{i=1}^n (z^{\alpha_i}-1)/(z^{\beta_i}-1) 
     \qquad (0< {\rm Arg} z <\pi)  \\
\intertext{and}
     \lim_{z\in \mathbb{H}_+ \to t} f(z) = f(t)  \quad \text{ for all } t>0.
\end{gather*}
Since the imaginary part $\Im f(z)$ of $f(z)$ is harmonic and positive on $\mathbb{H}_+$, 
$f(z)$ does not have a zero on $\mathbb{H}_+$.
This means that $|\gamma|\le 2$ and $0<\alpha_i \le 2$ ($i=1,2,\ldots,n$).
We also have $0<\beta_i \le 2$ ($i=1,2,\ldots,n$) because $f(z)$ does not have
a singular point on $\mathbb{H}_+$.

So we consider the problem that $f(t)$ becomes a Pick function on $(0, \infty)$
under the condition $|\gamma|<2$, $0<\alpha_i, \beta_i<2$
and $\alpha_i \neq \beta_j$   ($i,j=1,2,\ldots,n$).
In this case, the function
\[  f(z) = z^\gamma \prod_{i=1}^n (z^{\alpha_i}-1)/(z^{\beta_i}-1) \]
is holomorphic on $\mathbb{H}_+$ and continuous on the closure of $\mathbb{H}_+$.

We denote by ${\rm Arg}$ the function from $\mathbb{C}\setminus \{ 0 \}$ to
$[0, 2\pi)$ with $z = |z| e^{i{\rm Arg}z}$ for $z\in \mathbb{C}\setminus \{ 0 \}$.
When $|\gamma|\le 2, 0<\alpha_i, \beta_i\le 2$, we define $\arg f(z)$  as follows:
\[   \arg f(z) = \gamma {\rm Arg} z 
                    + \sum_{i=1}^n ({\rm Arg}(z^{\alpha_i} -1) - {\rm Arg}(z^{\beta_i}-1)),
     \quad z\in \mathbb{H}_+.  \]
Then we remark that, for $t>0$,
\[   \lim_{z\in \mathbb{H}_+\to t}\arg f(z) = 0 .  \]
We may consider that $\arg f(z)$ is continuous on the closure 
$\overline{\mathbb{H}_+}$ of $\mathbb{H}_+$ except $\{ 0 \}$
if $|\gamma|<2$ and $0<\alpha_i, \beta_i<2$.                                           

We define a function $F:[0,2]\times [0,2]\longrightarrow \mathbb{R}$ as follows:
\[  F(a,b) = \begin{cases}  a-b \quad & \text{if } a\ge b, \; 0\le b \le 1 \\
                  a-1 & \text{if } 1<a,b\le 2 \\
                  0  & \text{if } a<b, \; 0\le a \le 1 \end{cases}.  \]
Then we can prove the following statement:

\begin{theorem}
Let $|\gamma|\le 2$, $0<\alpha_i,\beta_i \le 2$, $\alpha_i\neq \beta_j$
$(i,j=1,2, \ldots , n)$ and $\alpha_i \le \alpha_j $,
$\beta_i \le \beta_j$ if $1\le i<j \le n$.
If it satisfies
\[  0 \le \gamma - \sum_{i=1}^n F(\beta_i, \alpha_i)
    \text{ and }
    \gamma + \sum_{i=1}^n F(\alpha_i, \beta_i) \le 1, \]
then we have
\[ f(t) =\begin{cases}
    t^\gamma \prod_{i=1}^n (t^{\alpha_i} -1)/(t^{\beta_i} -1) \qquad & \text{ if } t\neq 1\\
    \prod_{i=1}^n \alpha_i/\beta_i  & \text{ if } t=1 \end{cases}  \]
is operator monotone on $(0,\infty)$.

\end{theorem}

This theorem implies the following statement as seen in Szab\'o's paper\cite{szabo}:
$f$ is operator monotone if  
\[   0\le \gamma +\sum_{i=1}^u\alpha_i +(n-u)-(v+\sum_{j=v+1}^n \beta_j)
      \le \gamma + u  +\sum_{i=u+1}^n\alpha_i -\sum_{j=1}^v \beta_j -(n-v) \le 1, \]
where  $\alpha_1\le \cdots \le \alpha_u\le 1< \alpha_{u+1}\le \cdots \le \alpha_n$ 
and $\beta_1\le \cdots \le \beta_v\le 1 < \beta_{v+1}\le \cdots \le \beta_n$.
As examples we can show that 
\[    f_a(t)=a(1-a)\frac{(t-1)^2}{(t^a-1)(t^{1-a}-1)} \qquad (-1\le a \le 2) \]
is an operator monotone function on $[0, \infty)$ and
\[    c(\lambda, \mu) =(\lambda \mu)^{-1/2}\prod_{i=1}^n 
                     \frac{\sinh (\frac{\beta_i}{2} \log \frac{\lambda}{\mu})}
                     {\sinh (\frac{\alpha_i}{2} \log \frac{\lambda}{\mu})}  \]
is a Morozova-Chentsov function
(\cite{hiaipetz},\cite{kumagai}, \cite{morozovachentsov} and \cite{petz}) 
if it satisfies
\[   0\le \gamma +\sum_{i=1}^u\alpha_i+(n-u) -(v+\sum_{j=v+1}^n \beta_j)
      \le \gamma + u +\sum_{i=u+1}^n\alpha_i -\sum_{j=1}^v \beta_j -(n-v) \le 1, \]
where $\gamma = (1+\sum_{i=1}^n(\alpha_i-\beta_i) )/2$.


\section{Proof of theorem}

\begin{lemma}
Let $0<b<a<2$ and $z=re^{i\theta}$ $(0\le \theta={\rm Arg} z \le \pi)$.
For any $\epsilon > 0$, there exists an $R>0$ such that 
\[  r>R \Rightarrow |\arg \frac{z^a-1}{z^b-1} - (a-b)\theta|<\epsilon .  \]
\end{lemma}
\begin{proof}
When $|z|\to \infty$, $(z^a-1)/z^a = 1 - 1/z^a \to 1$ and
$z^b/(z^b-1)=1 + 1/(z^b-1) \to 1$.
So we can choose an $R>0$ such that 
$|\arg (z^a-1)/z^a|, |\arg z^b/(z^b-1)|<\epsilon/2$.
Then we have
\[   |\arg \frac{z^a-1}{z^b-1} - (a-b)\theta|
     = | \arg \frac{z^a-1}{z^a}\frac{z^a}{z^b}\frac{z^b}{z^b-1} -
          \arg \frac{z^a}{z^b}|<\epsilon . \]
\end{proof}

\vspace{5mm}

Let $\alpha = (\alpha_1,\alpha_2,\ldots , \alpha_n)$ and 
$\beta = (\beta_1, \beta_2, \ldots, \beta_n)$ with
\[   0<\alpha_i, \beta_i<2  \text{  and } \alpha_i\neq \beta_j \quad (i,j=1,2,\ldots , n) \]
and set
\[  g(z) =\begin{cases} \prod_{i=1}^n (z^{\alpha_i}-1)/(z^{\beta_i}-1) \quad 
      & \Im z\ge 0 \text{ and } z \neq 1\\
      \prod_{i=1}^n \alpha_i/\beta_i & z=1 \end{cases} . \]
We define numbers $\theta(\alpha, \beta)$ and $\Theta(\alpha, \beta)$ as follows:
\begin{gather*}
    \theta(\alpha, \beta) = \inf \{ \arg g(re^{\pi i}) \mid 0<r\le 1 \}  \\
    \text{and }  \Theta(\alpha, \beta) = \sup \{ \arg g(re^{\pi i}) \mid 0< r\le 1 \}  .
\end{gather*}
Since $\lim_{r\to 0+} \arg g(re^{\pi i}) = 0$, we have 
$\theta(\alpha, \beta)\le 0$.
Remarking the fact
\[  \frac{(re^{\pi i})^a-1}{(re^{\pi i})^b -1}= (re^{\pi i})^{a-b} \frac{1-(e^{-\pi i}/r)^a}{1-(e^{-\pi i}/r)^b}
      =  (re^{\pi i})^{a-b}  \overline{ (\frac{(e^{\pi i}/r)^a-1}{(e^{\pi i}/r)^b-1}) } ,   \]
we can get the following relation:
\[  \arg g(e^{\pi i}/r) + \arg g(re^{\pi i}) = \sum_{i=1}^n (\alpha_i - \beta_i) \pi .  \]
Then the following numbers $F_0(\alpha, \beta)$ and $G_0(\alpha, \beta)$ can be represented by
$\theta(\alpha,\beta)$ and $\Theta(\alpha, \beta)$ as follows:
\begin{align*}
  F_0(\alpha, \beta) & = \sup \{ \arg g(re^{\pi i})/\pi \mid r>0 \} \\
            & = \max \{ \Theta(\alpha, \beta), \sum_{i=1}^n (\alpha_i-\beta_i)\pi - \theta(\alpha,\beta) \} /\pi 
\end{align*}
and
\begin{align*}
  G_0(\alpha, \beta) & = \inf \{\arg g(re^{\pi i})/\pi \mid r>0 \}  \\
          &   = \min \{ \theta(\alpha, \beta), \sum_{i=1}^n (\alpha_i-\beta_i)\pi - \Theta(\alpha, \beta) \} /\pi.
\end{align*}
We enumerate the facts which follow from above observation:
\begin{gather*}
   \lim_{r\to 0+} \arg g(re^{\pi i}) =0, \quad
   \lim_{r\to \infty} \arg g(re^{\pi i}) = \sum_{i=1}^n (\alpha_i-\beta_i) \pi , \\
   G_0(\alpha, \beta)\le \min \{0, \sum_{i=1}^n(\alpha_i - \beta_i)\},  \quad 
   F_0(\alpha, \beta) \ge \max \{0, \sum_{i=1}^n(\alpha_i - \beta_i) \} , \\
   G_0(\alpha, \beta) < 0 \; \Leftrightarrow \; 
   F_0(\alpha, \beta) > \sum_{i=1}^n (\alpha_i - \beta_i)
\intertext{and}
   G_0(\alpha, \beta)= -F_0(\beta, \alpha)  .
\end{gather*}

\vspace{5mm}

%
%

\begin{theorem}
Let $|\gamma|\le 2$, $0<\alpha_i, \beta_i < 2$,  $\alpha_i\neq \beta_j$ 
$(i,j=1,2,\ldots, n)$
and
\[   f(t) =\begin{cases} t^\gamma \prod_{i=1}^n (t^{\alpha_i}-1)/(t^{\beta_i}-1)
       \quad & t\in (0,\infty)\setminus\{1\} \\
       \prod_{i=1}^n \alpha_i/\beta_i & t=1 \end{cases} . \]
Then, for any $s>0$, the following are equivalent:
\begin{enumerate}
  \item[$(1)$]  $f(t)^s$ is operator monotone.
  \item[$(2)$]  $\gamma -F_0(\beta,\alpha)\ge0$ and
                 $s(\gamma + F_0(\alpha, \beta)) \le 1$.
\end{enumerate}
\end{theorem}
\begin{proof}
(1)$\Rightarrow$(2)\;
Since $f(t)^s$ is operator monotone, this implies
\begin{align*}
      & \quad 0 \le \arg f(re^{\pi i})^s \le \pi  \\
  \Rightarrow & \quad 0\le s( \gamma \pi +\arg g(re^{\pi i})) \le \pi \\
  \Rightarrow & \quad 0\le \gamma +G_0(\alpha, \beta) \text{ and } s(\gamma + F_0(\alpha, \beta)) \le 1 \\
  \Rightarrow & \quad 0\le \gamma -F_0(\beta, \alpha) \text{ and } s(\gamma + F_0(\alpha, \beta)) \le 1.
\end{align*}

(2)$\Rightarrow$(1)\; When $z=re^{\pi i}$ $(r>0$), we have
\[  s(\gamma-F_0(\beta, \alpha)) \pi = s(\gamma+G_0(\alpha, \beta)) \pi
   \le \arg (f(z)^s) \le s(\gamma + F_0(\alpha, \beta))\pi .  \]
and, by assumption,
\[  0 \le \arg (f(z)^s) \le \pi.  \]

We remark that
\[   0\le s F_0(\beta, \alpha) \le s\gamma \le 1-s F_0(\alpha, \beta)\le 1 . \]
Since
\[  \lim_{|z|\to 0} \arg g(z) = \lim_{|z|\to 0}\arg(\prod_{i=1}^n 
         \frac{(z^{\alpha_i}-1)}{(z^{\beta_i}-1)} )= 0,  \]
 for any $\epsilon >0$, we can choose $\delta_0 >0$ such that $0<|z| \le \delta_0$
implies 
\[  |\arg f(z)^s - s\gamma {\rm Arg} z|<\epsilon.  \]
In particular, we have $-\epsilon < \arg (f(z)^s) < \pi +\epsilon$ for any $z\in \mathbb{H}_+$
with $0<|z|\le \delta_0$.

To prove that $f(t)^s$ is operator monotone, it suffices to show that
\[  0 \le \arg (f(z)^s) \le \pi \qquad \text{ for all } z\in \mathbb{H}_+.  \]
We assume that there exists $z_0\in \mathbb{H}_+$ such that $\arg (f(z_0)^s)<0$ or 
$\arg (f(z_0)^s)>\pi$.
For arbitrary $\epsilon >0$, by Lemma 2.1, we can choose $R_0$ such that $R_0>|z_0|$ and
the condition $R>R_0$ implies
\[  (\gamma + \sum_{i=1}^n(\alpha_i-\beta_i)){\rm Arg} z -\epsilon < \arg f(z) <
     (\gamma + \sum_{i=1}^n(\alpha_i-\beta_i)){\rm Arg} z +\epsilon  \]
for all $z\in \mathbb{H}_+$ with $|z|=R$.
When $\arg (f(z_0)^s)<0$, we choose $\epsilon >0$ with  $\arg (f(z_0)^s)<\min \{ -\epsilon, -s\epsilon\}$.
It holds $|z_0|>\delta_0$ automatically, 
so we have $\arg (f(z_0)^s) <\arg (f(z)^s)$ if  $z\in \mathbb{H}_+$ with 
$|z|=R$ or $|z|=\delta_0$.
This contradicts to the maximum value principle of the harmonic function 
$\arg( f(\cdot)^s)\left( = s \arg(f(\cdot)) \right) $ for the closed region
$\{z \mid \Im z\ge 0, \delta_0 \le |z|\le R\}$.
When $\arg (f(z_0)^s)>\pi$, we choose $\epsilon >0$ with  
$\arg (f(z_0)^s)>\pi +s\epsilon$.
Since $\sum_{i=1}^n(\alpha_i-\beta_i)\le F_0(\alpha, \beta)$, we have
 $\arg (f(z_0)^s) >\arg (f(z)^s)$ if  $z\in \mathbb{H}_+$ with $|z|=R$.
This also implies a contradiction by the similar reason.
\end{proof}

We put $s=1$ and $\gamma = 0$ in the above proof of (2)$\Rightarrow$(1).
Because $G_0(\alpha, \beta)\le 0$ and $F_0(\alpha,\beta) \ge \sum_{i=1}^n(\alpha_i-\beta_i)$,
we can get the following from this argument:

\vspace{5mm}

%
%

\begin{corollary}
Let $0<\alpha_i, \beta_i < 2$,  $\alpha_i\neq \beta_j$ $(i,j=1,2,\ldots, n)$.
Then we have
\[   -F_0(\beta, \alpha)\pi =G_0(\alpha, \beta)\pi \le \arg g(z) \le F_0(\alpha, \beta) \pi \]
for all $z$ with $\Im z\ge 0$.

Moreover $g$ is an operator monotone function on $[0,\infty)$ if and only if
$0\le G_0(\alpha,\beta)\le F_0(\alpha,\beta) \le 1$.
\end{corollary}

\vspace{5mm}

%
%

\begin{lemma}
Let $z=re^{i\pi}$ $(r>0)$ and $0<b<a<2$.
\begin{enumerate}
  \item[$(1)$] If $b>1$, then we have
\[   (1-b)\pi \le \arg \frac{z^a-1}{z^b-1} \le (a-1)\pi  .  \]
  \item[$(2)$] If $b \le 1$, then we have
\[   0  \le \arg \frac{z^a-1}{z^b-1} \le (a-b)\pi  . \]
\end{enumerate}
\end{lemma}
\begin{proof}
(1) \; Since $\pi < {\rm Arg}(z^a-1) < {\rm Arg} z^a = a \pi$ and
$\pi < {\rm Arg}(z^b-1) < {\rm Arg} z^b = b \pi$,
we have 
\[   (1-b)\pi < \arg \frac{z^a-1}{z^b-1} < (a-1)\pi .  \]

(2) \; When $b\le 1<a$, we have 
\[   0 < \arg \frac{z^a-1}{z^b-1} < (a-b)\pi ,  \]
since  $\pi < {\rm Arg}(z^a-1) < {\rm Arg} z^a = a \pi$ and
$b \pi = {\rm Arg} z^b \le  {\rm Arg}(z^b-1) \le  \pi$.

When $a\le 1$, we consider the case that there exist positive integers $k, l$ ($k>l$) 
such that $z^a=w^k$ (i.e., $w=r^{a/k}e^{a\pi i/k}$) and $z^b=w^l$.
Since
\[  \frac{z^a-1}{z^b-1}=\frac{w^{k-1}+\cdots + w+1}{w^{l-1}+\cdots+w+1}
         =z^{a-b} \cdot \frac{w^{k-1}+\cdots + w + 1}{w^{k-1}+\cdots+w^{k-l+1}+w^{k-l}} \]
and $\arg (w^{k-1} + w^{k-2} +\cdots +w^m) \ge  
\arg(w^{m-1} + \cdots + w + 1)$ for any $m=1,2,\ldots, k-1$, we have
\[  0  \le \arg \frac{z^a-1}{z^b-1} \le (a-b)\pi  . \]
By the continuity for $a$ and $b$, it also holds for any $a, b$ with $a\le 1$.
\end{proof}

\vspace{5mm}

%
%

\begin{proposition}
Let $0<a, b \le 2$ and 
$z\in \mathbb{H}_+$.
Then we have
\[  -F(b, a) \pi \le \arg \frac{z^a -1}{z^b -1} 
                             \le F(a, b) \pi ,  \]
where
\[  F(a, b) = \begin{cases} a-b \quad & \text{if } a\ge b, \; 0\le b \le 1 \\
                                                a-1 & \text{if } 1<a,b \le 2  \\
                                                0  & \text{if } a<b, \; 0\le a \le 1 \end{cases}. \]
\end{proposition}
\begin{proof}
We assume that $0<b<a<2$.
Since $\arg \frac{z^a-1}{z^b-1} = -\arg \frac{z^b-1}{z^a-1}$,
it suffices to show that
\begin{align*}
  0 \le & \arg \frac{z^a-1}{z^b-1} \le (a-b)\pi &
                    \text{if  } & 0<b\le 1,  \\
  (1-b)\pi \le & \arg \frac{z^a-1}{z^b-1} \le (a-1) \pi &
                    \text{if  } & 1<b<2.
\end{align*}
It is clear from Lemma 2.4 and Corollary 2.3.

We consider the case $a=2$ or $b=2$.
Choose sequences $\{a_n\}, \{b_n\}$ with
\[  0<a_n, b_n<2 \text{ and } \lim_{n\to \infty}a_n = \lim_{n\to \infty} b_n=2, \]
we have already shown the statement for $0<a_n,b_n<2$.
Taking the limit,  it holds for the case $a=2$ or $b=2$.
\end{proof}

\vspace{5mm}

%
%

\begin{theorem}
Let $|\gamma|\le 2$, $0<\alpha_i,\beta_i \le 2$ and $\alpha_i\neq \beta_j$
$(i,j=1,2, \ldots , n)$.
We denote by $S_n$ the set of all permutations on $\{1,2,\ldots,n\}$.
If it satisfies
\[  0 \le \gamma - \min_{\sigma\in S_n} \sum_{i=1}^n F(\beta_{\sigma (i)}, \alpha_i)
    \text{ and }
    \gamma + \min_{\sigma\in S_n} \sum_{i=1}^n F(\alpha_i, \beta_{\sigma (i)}) \le 1, \]
then we have
\[ f(t) =\begin{cases}
    t^\gamma \prod_{i=1}^n (t^{\alpha_i} -1)/(t^{\beta_i} -1) \qquad & \text{ if } t\neq 1\\
    \prod_{i=1}^n \alpha_i/\beta_i  & \text{ if } t=1 \end{cases} \]
is operator monotone on $(0,\infty)$.

\end{theorem}
\begin{proof}
We may assume that $0<\alpha_i, \beta_i <2$.
Let $\sigma, \tau \in S_n$ with
\[  \gamma - \sum_{i=1}^n F(\beta_{\sigma(i)}, \alpha_i) \ge 0, \quad
     \text{ and } \quad
     \gamma + \sum_{i=1}^n F(\alpha_i, \beta_{\tau(i)}) \le 1. \]


When $z=re^{i\pi}$ $(r>0)$, we have 
\begin{align*}
  -\sum_{i=1}^n F(\beta_{\sigma(i)}, \alpha_i))\pi \le & 
        \sum_{i=1}^n \arg \frac{z^{\alpha_i}-1}{z^{\beta_{\sigma(i)}}-1}
        = \arg \prod_{i=1}^n \frac{z^{\alpha_i} -1}{z^{\beta_i}-1}  \\
      = & \sum_{i=1}^n \arg \frac{z^{\alpha_i}-1}{z^{\beta_{\tau(i)}}-1}
      \le \sum_{i=1}^n F(\alpha_i, \beta_{\tau(i)}) \pi
\end{align*}
by Proposition 2.5.
This means that
\[  -\sum_{i=1}^n F(\beta_{\sigma(i)}, \alpha_i)) \le G_0(\alpha, \beta) 
    = -F_0(\beta, \alpha), 
     F_0(\alpha, \beta) \le \sum_{i=1}^n F(\alpha_i, \beta_{\tau(i)}) \]
and
\[  0 \le \gamma -F_0(\beta, \alpha) \le \gamma + F_0(\alpha, \beta) \le 1. \]
So $f$ is operator monotone by Theorem 2.2.
%
\end{proof}

\vspace{5mm}

We remark that it holds
\[  -\sum_{i=1}^n F(\beta_i, \alpha_i) \le G_0(\alpha, \beta) \le
      F_0(\alpha, \beta) \le \sum_{i=1}^n F(\alpha_i, \beta_i)  \]
from the argument in this proof.

For $0\le a_1 \le a_2 \le 2$ and $0 \le b_1 \le b_2 \le 2$, we have
\[  F(a_1,b_1) + F(a_2,b_2) \le F(a_1,b_2) + F(a_2,b_1) . \] 
To prove this, we set $D = ( F(a_1,b_2) + F(a_2,b_1) ) - (F(a_1,b_1) + F(a_2,b_2))$
and show $D\ge 0$ in each of the following cases:
\begin{align*}
(1) & \; a_1\ge 1  &  (2) & \; a_1 \le 1 \le a_2, \; b_1\ge1 \\ 
(3) & \; a_1 \le b_1 \le1 \le a_2  & (4) & \;b_1 \le a_1 \le 1 \le a_2,\; 1\le b_2 \\
(5) & \; b_1 \le a_1 \le 1 \le a_2,\; a_1\le b_2\le 1 &   (6) & \; b_2 \le a_1\le 1 \le a_2 \\
(7) & \; a_2\le 1, \; a_2 \le b_1 &  (8) & \;  a_2 \le 1, \; a_1 \le b_1 \le a_2 \le b_2 \\
(9) & \;  a_2 \le 1, \; b_1\le a_1\le a_2 \le b_2  &  (10) &  \;  a_1 \le b_1 \le b_2 \le a_2 \le 1  \\
(11) &  \;  b_1 \le a_1 \le b_2 \le a_2 \le 1 & (12) &  \;  b_2 \le a_1, a_2 \le 1
\end{align*}

Case (1).  Since $(F(a_1,b_2)-F(a_1,b_1)) = (F(a_2,b_2)-F(a_2,b_1))$, $D =0$.

Case (3). $D =  (0 + (a_2- b_1)) - (0 +F(a_2, b_2)) \ge 0$. \\
There are many easy calculations to show $D\ge 0$ in the rest cases
(In particular, $D=0$ in the cases (1), (2), (6), (7) and (12)).
So we omit them.

By using the property
\[  F(a_1,b_1) + F(a_2,b_2) \le F(a_1,b_2) + F(a_2,b_1) \qquad (a_1\le a_2, \; b_1\le b_2),\] 
we can get the following statement:

\vspace{5mm}

\begin{proposition}
Let $a_1,a_2,\ldots,a_n,b_1,b_2,\ldots,b_n\in [0,2]$ and $\sigma, \tau \in S_n$  permutations on $\{1,2,\ldots,n\}$ with 
\[     a_{\sigma(i)} \le a_{\sigma(j)}, \; b_{\tau(i)} \le b_{\tau(j)} \qquad
   \text{if } i<j .  \]
Then we have
\[   \sum_{i=1}^n F(a_{\sigma(i)}, b_{\tau(i)}) \le \sum_{i=1}^n F(a_i, b_i).  \]
\end{proposition}
\begin{proof}
We use an induction for $n$.
In the case $n=2$, we have already proved in above remark.

We assume that it holds for $n$, and will show that it also holds for $n+1$.
We may assume that
\[  a_1\le a_2 \le \ldots \le a_n \le a_{n+1}. \]
We choose $k$ with $b_k =\max \{b_1,b_2,\ldots,b_{n+1}\}$.
There is a permutation $\tau \in S_{n+1}$ such that
\[  \tau(i) = \begin{cases} n+1  \quad & i= k  \\
                                   k   & i=n+1 \\
                                   i   & \text{otherwise} \end{cases} .  \]
Since $a_k\le a_{n+1}$ and $b_{n+1}\le b_k$, 
\begin{align*}
  F(a_k,b_{\tau(k)}) + F(a_{n+1}, b_{\tau(n+1)}) & = F(a_k, b_{n+1})+F(a_{n+1}, b_k) \\ 
     &  \le F(a_k,b_k)+F(a_{n+1}, b_{n+1}).  
\end{align*}
Then we have
\begin{align*}
 & \sum_{i=1}^{n+1} F(a_i, b_i) 
  = \sum_{i \neq k, n+1} F(a_i, b_i) + F(a_k,b_k) + F(a_{n+1}, b_{n+1}) \\
 \ge & \sum_{i \neq k, n+1} F(a_i, b_i) + F(a_k,b_{\tau(k)}) + F(a_{n+1}, b_{\tau(n+1)})  
 = \sum_{i=1}^{n+1} F(a_i, b_{\tau(i)})  .
\end{align*}
By applying the hypothesis of induction for $a_1,\ldots, a_n$ and $b_1,\ldots,b_{k-1},
b_{k+1},\ldots, b_{n+1}$, there exists a permutation $\sigma \in S_{n+1}$ such that
\begin{gather*}
  \sigma ( \{1,2,\ldots, n\} ) = \{1,2,\ldots ,n+1\} \setminus \{k \}, \;
   \sigma(n+1) = k  \\
  b_{\sigma(i)}\le b_{\sigma(j)} (\le b_k) \quad \text{if } 1\le i <j \le n  \\
\intertext{and }
   \sum_{i=1}^n F(a_i, b_{\sigma(i)}) \le \sum_{i=1}^n F(a_i, b_{\tau(i)}) .
\end{gather*}
So it follows that
\begin{align*}
    \sum_{i=1}^{n+1} F(a_i, b_{\sigma(i)})  
         &   \le \sum_{i=1}^n F(a_i, b_{\tau(i)}) + F(a_{n+1}, b_k)  \\
         &   = \sum_{i=1}^n F(a_i, b_{\tau(i)}) + F(a_{n+1}, b_{\tau(n+1)})
            \le \sum_{i=1}^{n+1} F(a_i, b_i).
\end{align*}
\end{proof}

\vspace{5mm}

We can see that Theorem 2.6 implies Theorem 1.1 and also that
these two theorems are equivalent by Proposition 2.7.

In \cite{szabo}  Szab\'o remarked that
\begin{gather*}
   a \pi \le {\rm Arg} (z^a-1) \le \pi \qquad 0\le a \le 1, \\
   \pi \le {\rm Arg} (z^a-1) \le a \pi \qquad 1\le a \le 2,
\end{gather*}
when $z = re^{i\pi}$  $(r\ge 0)$.
This means that
\[  -S(b,a)\pi \le \arg\frac{z^a-1}{z^b-1} \le S(a,b) \pi   \]
for any $z = re^{i\pi}$, where
\[  S(a,b) = \begin{cases} 1-b \quad & 0 < a,b \le 1 \\
                0  & 0< a \le 1 \le b \le 2 \\
                a-b  & 0 < b \le 1 \le a \le 2 \\
                a-1  & 1\le a,b \le 2 \end{cases} .  \]
Since $a-b\le F_0(a,b) \le F(a,b) \le S(a,b)$, it follows from Theorem 1.1 that,
for $|\gamma|\le 2$ and  $0<\alpha_i, \beta_j\le 2$ $(1\le i,j \le n)$ with
$\alpha_1\le \cdots \le \alpha_u\le 1 <\alpha_{u+1}\le \cdots \le \alpha_n$ and
$\beta_1 \le \cdots \le \beta_v \le 1 < \beta_{v+1} \le \cdots \le \beta_n$,
\[ f(t) =\begin{cases}
    t^\gamma \prod_{i=1}^n (t^{\alpha_i} -1)/(t^{\beta_i} -1) \qquad & \text{ if } t\neq 1\\
    \prod_{i=1}^n \alpha_i/\beta_i  & \text{ if } t=1 \end{cases} \]
is operator monotone on $[0,\infty)$  if
\[     0\le \gamma - \sum_{i=1}^n S(b_i,a_i) \le
          \gamma + \sum_{i=1}^n S(a_i,b_i)\le 1 .  \]  
In other words, $f(t)$ becomes operator monotone if
\[    0\le \gamma +\sum_{i=1}^u\alpha_i +(n-u)-(v+\sum_{j=v+1}^n \beta_j)
      \le \gamma + u  +\sum_{i=u+1}^n\alpha_i -\sum_{j=1}^v \beta_j -(n-v) \le 1. \]
When $\gamma = (1+ \sum_{i=1}^n (\alpha_i-\beta_i))/2$, it holds $f(t) = tf(1/t)$
and $c(\lambda, \mu) = \frac{1}{\mu f(\lambda \mu^{-1})}$ becomes Morozowa-Chentsov function (\cite{szabo}).

\vspace{5mm}


\section{Additional results}

Let  $\alpha =(\alpha_1,\alpha_2,\ldots,\alpha_n)$,
$\beta =(\beta_1, \beta_2, \ldots,\beta_n)$ with
$\alpha_i, \beta_i >0$ and $\alpha_i\neq \beta_j$  $(i,j=1.2.\ldots,n)$.
We set
\[  g(z) =\begin{cases} \prod_{i=1}^n (z^{\alpha_i}-1)/(z^{\beta_i}-1) \quad 
      & \Im z\ge 0 \text{ and } z \neq 1\\
      \prod_{i=1}^n \alpha_i/\beta_i & z=1 \end{cases} . \]

\vspace{5mm}

%
%

\begin{proposition}
If $g$ is holomorphic on $\mathbb{H}_+$ and has no zeros on $\mathbb{H}_+$, then $\max\{ \alpha_i, \beta_j \mid i,j=1,2,\ldots,n \}\le 2$.

In particular, if $\max\{ \alpha_i, \beta_j \mid i,j=1,2,\ldots,n \}>2$, then $g(t)$ is not operator monotone on $[0,\infty)$. 
\end{proposition}
\begin{proof}
In the case of $\alpha_1>2$.
Then $e^{\pi i/\alpha_1} \in \mathbb{H}_+$ and $(e^{\pi i/\alpha_1} )^{\alpha_1}-1 = 0$.
Because $g$ has no zeros on $\mathbb{H}_+$, there exists some $\beta_j$ satisfying $(e^{\pi i/\alpha_1})^{\beta_j} -1 = 0$,
that is, $\beta_j \ge \alpha_1 >2$.

In the case of $\beta_1>2$.
Then $e^{\pi i/\beta_1} \in \mathbb{H}_+$ and $(e^{\pi i/\beta_1})^{\beta_1}-1=0$.
Since $g(z)$ is holomorphic on $\mathbb{H}_+$,  $e^{\pi i/\beta_1}$ is a removable singularity of $g(z)$.
This means that $(e^{\pi i/\beta_1})^{\alpha_k}-1=0$ for some $\alpha_k$, that is, $\alpha_k\ge \beta_1>2$.

If $\max\{ \alpha_i, \beta_j \mid i,j=1,2,\ldots,n \}>2$, then it contradicts to $\alpha_i\neq \beta_j$ for any  $i,j=1,2,\ldots, n$.
\end{proof}

\vspace{5mm}

In the rest of this section, we assume that $0\le \alpha_i, \beta_i \le 2$.
We have already shown that
\begin{gather*}
  G_0(\alpha, \beta)\le 0, \quad F_0(\alpha, \beta)
       \ge \sum_{i=1}^n(\alpha_i-\beta_i) \\
  \text{and }  G_0(\alpha, \beta)<0 \;  \Leftrightarrow \; 
  F_0(\alpha,\beta)>\sum_{i=1}^n (\alpha_i - \beta_i),
\intertext{where}
  F_0(\alpha, \beta)  = \sup \{ \arg g(re^{\pi i}) / \pi \mid r>0 \}
         = \sup \{ \arg g(z) / \pi \mid z\in \mathbb{H}_+ \}  \\
  \text{and } G_0(\alpha, \beta) = -F_0(\beta, \alpha).
\end{gather*}

\vspace{5mm}

%
%

\begin{proposition}
Let $0<b<a \le 2$.
Then the following are equivalent:
\begin{enumerate}
  \item[$(1)$] $F_0(a,b)=a-b$.
  \item[$(2)$] $G_0(a,b) =0$.
  \item[$(3)$] $0<b\le 1$.
\end{enumerate}
\end{proposition}
\begin{proof}
$(1)\Leftrightarrow (2)$ It has been already stated in the above remark.

$(3)\Rightarrow (2)$ Because $F(b,a)=0$, 
\[  0 \ge G_0(a,b) = -F_0(b,a) \ge - F(b,a) =0.  \]

$(2)\Rightarrow (3)$ It suffices to show that $G_0(a,b)<0$ in the case 
$1< b < 2$.
We set $z=r e^{i\theta}$ and
\begin{align*}
  g(z) & = \frac{z^a-1}{z^b-1} =\frac{1}{|z^b -1|^2}(z^a-1)(\bar{z}^b-1) \\
        & = \frac{1}{|z^b -1|^2} (r^{a+b}\cos (a-b)\theta 
                    - r^a \cos a\theta -r^b \cos b\theta +1) \\
        & \qquad + \frac{i}{|z^b -1|^2} (r^{a+b}\sin (a-b)\theta
                    - r^a \sin a\theta + r^b \sin b\theta).
\end{align*}
When 
$1<b<2$, we choose a number $\theta$ satisfying
\[  \frac{\pi}{b} < \theta < \pi .  \]
For a sufficiently small positive $r$, we also have
\[  \Re g(z) >0  \text{ and } \Im g(z)<0 .  \]
This means that $G_0(a,b)<0$ if 
$1<b<2$.
\end{proof}

\vspace{5mm}

%
%

\begin{example}
\begin{enumerate}
  \item[$(1)$] Let $\alpha_i>\beta_i$ and $\beta_i \le 1$ $(i=1,2,\ldots, n)$.
Then we have
\[  F_0(\alpha, \beta)= \sum_{i=1}^n(\alpha_i - \beta_i) \text{ and }
    G_0(\alpha, \beta) = 0.  \]
Moreover,  $g(t)= \prod_{i=1}^n (t^{\alpha_i}-1)/(t^{\beta_i}-1)$ is operator monotone if and only if 
$0\le \sum_{i=1}^n(\alpha_i-\beta_i) \le 1$.
  \item[$(2)$] For real numbers $a$, $b$ with $|a|,|b|\le 2$ and $a\neq b$,
we define the function $h_1:(0,\infty) \longrightarrow \mathbb{R}$ as follows:
\[  h_1(t) = \frac{b}{a}\frac{t^a-1}{t^b-1},   \]
where we consider $(t^a-1)/a$ as $\log t$ in the case of $a=0$.
Then $h_1$ is operator monotone on $(0,\infty)$ if and only if 
\begin{align*}
   (a,b) \in & \{ (a,b)\in \mathbb{R}^2 \mid 0< a-b \le 1, \; a \ge -1, \text{ and }
    b \le 1 \} \\
   & \qquad \cup ( [0,1]\times [-1,0]) \setminus \{(0,0)\}.  
\end{align*}
  \item[$(3)$] For real numbers $a$, $b$ with $a\neq b$,
we define the function $h_2:[0,\infty) \longrightarrow \mathbb{R}$ as follows:
\[  h_2(t) = \frac{t^a+1}{t^b+1}.   \]
If $h_2$ is operator monotone on $[0,\infty)$, then  $(a,b)\in [0,1]\times [-1,0]$. 
In particular, if $( 0<a-b \le 1$ and $b \le 0 \le a)$ or $(0<a=-b\le 1)$, then
$h_2$ is operator monotone, and if $(a=1$ and $-1<b<0)$ or $(b=-1$ and $0<a<1)$, then
$h_2$ is not operator monotone.
\end{enumerate}
\end{example}
\begin{proof} \;
(1) By Proposition 3.2, we have
\[  \sum_{i=1}^n (\alpha_i-\beta_i) = \sum_{i=1}^n F_0(\alpha_i, \beta_i)
     \ge F_0(\alpha, \beta) \ge \sum_{i=1}^n (\alpha_i - \beta_i) . \]
This means that $F_0(\alpha, \beta)= \sum_{i=1}^n(\alpha_i - \beta_i)$ and
$G_0(\alpha, \beta) = 0$.

The rest part follows from Corollary 2.3.

(2) When $a,b\ge 0$, we have $\lim_{r\to \infty} \arg h_1(r e^{\pi i}) = (a-b) \pi$.
Since the operator monotonicity of $h_1$ implies $b<a$,
we have that $h_1$ is operator monotone if and only if 
$0<a-b \le 1$ and $0\le b \le 1$ by Proposition 3.2.

Assume that $a \ge 0 > b$.
If $h_1$ is operator monotone, then we have $(a,b) \in [0,1]\times [-1,0]$ 
because
\[  \lim_{ r\to 0+} \arg h_1(re^{i\pi}) = -b \pi   \; \text{ and } \;
     \lim_{r\to \infty} \arg h_1(re^{i \pi}) = a \pi.  \]
The function $h_1$ can be written as follows:
\[  h_1(t) = \frac{-b}{a}\cdot t^{-b} \cdot \frac{t^a-1}{t^{-b}-1}.  \]
Since
\begin{gather*}
    -b-F(-b,a) =\begin{cases} -b  \quad & 0\le-b\le a \le 1 \\
                                        a   & 0 \le a < -b \le 1 \end{cases} \\
\intertext{and }
    -b+F(a,-b) =\begin{cases} a      & 0\le-b\le a \le 1 \\
                                       -b \quad  & 0 \le a < -b \le 1 \end{cases} ,
\end{gather*}
$h_1$ is operator monotone if and only if 
$(a,b) \in [0,1]\times [-1,0]$ by Theorem 2.6.

Assume that $b > 0 \ge a$.
Since $\lim_{r \to \infty} \arg h_1(re^{\pi i}) = -b \pi$, $h_1$ is not operator monotone.

Assume that  $a,b \le 0$. 
Since $\lim_{r \to 0+} \arg h_1(re^{\pi i}) = (a-b) \pi$, the operator monotonicity of $h_1$ 
implies  $0\le a-b \le 1$.
We rewrite the function $h_1$ as follows:
\[  h_1(t) =\frac{-b}{-a}\cdot t^{a-b}\cdot \frac{t^{-a}-1}{t^{-b}-1}.  \]
We assume that $a <-1$ and $0\le a-b \le 1$.
Because $1\le -a \le -b \le2$, we have $F_0(-b,-a) >(a-b)\pi$ by Proposition 3.2.
So there exists $z_0\in (-\infty, 0]$ such that
\[ \arg \frac{z_0^{-a}-1}{z_0^{-b}-1} < (b-a ) \pi  \text{ and }
   \arg h_1(z_0) <0,   \]
that is, $h_1$ is not operator monotone.
We assume that $a>-1$ and $0\le a-b \le 1$.
Since
\[  a-b -F(-b,-a) =0  \text{ and } a-b +F(-a,-b) =a-b \le 1, \]
$h_1$ is operator monotone.

(3) Since
\begin{align*}
  h_2(t) & = \frac{t^b-1}{t^a-1}\frac{t^{2a}-1}{t^{2b}-1}
          =t^{-b}\frac{t^{-b}-1}{t^a-1}\frac{t^{2a}-1}{t^{-2b}-1}  \\
       & =t^a\frac{t^b-1}{t^{-a}-1} \frac{t^{-2a}-1}{t^{2b}-1}
          =t^{a-b}\frac{t^{-b}-1}{t^{-a}-1}\frac{t^{-2a}-1}{t^{-2b}-1} , 
\end{align*}
$h_2$ is not operator monotone if 
$(a,b) \notin [-1,1]\times [-1,1]$ by Proposition 3.1.

For $z=re^{i\pi}$, we have
\begin{align*}
  h_2(z) & = \frac{1}{ | z^b+1 |^2} (z^a+1) (\bar{z}^b+1) \\
       & =  \frac{1}{ |z^b+1|^2} 
            (r^{a+b}\cos (a-b)\pi + r^a\cos a \pi + r^b \cos b\pi +1 ) \\
       & \qquad  +\frac{i}{|z^b+1|^2} 
            (r^{a+b}\sin (a-b)\pi + r^a \sin a \pi -r^b\sin b \pi).
\end{align*}

We assume $0\le a,b \le 1$.
Since $\lim_{r\to \infty} \arg h_2(re^{\pi i})= (a-b)\pi$, $h_2$ is not operator monotone when
$b>a$.
If $0<b<a \le 1$, then there exists a sufficiently small $r>0$ such that
\[  \Im h_2(re^{\pi i}) = \frac{r^b}{|z^b+1|^2}(r^a \sin (a-b)\pi + r^{a-b}\sin a\pi
       - \sin b\pi) <0,  \]
that is, $h_2$ is not operator monotone.

We assume $-1\le a, b<0$.
Applying the similar argument for the facts
\[
  \lim_{r\to 0+} \arg h_2(re^{\pi i}) =(a-b) \pi  
\]
and, for a sufficiently large $r>0$, $\Im h_2(re^{\pi i}) <0$ when 
$-1\le b <a <0$,
we can show  that $h_2$ is not operator monotone.

We assume $-1 \le a\le 0$ and $0<b \le 1$.
Since $\lim_{r\to \infty} \arg h_2(re^{\pi i}) = -b \pi <0$,
$h_2$ is also not operator monotone.

We assume $0\le a \le 1$ and $-1\le b \le 0$.
We can easily  show that, for $z=re^{i\theta}$ $(0\le \theta \le \pi)$,
$\{ \arg h_2(z) \}$ uniformly converges to $a\theta$ (resp. $-b\theta$) 
when $r$ tends to $\infty$ (resp. $r$ tends to $0$). 
If $0\le a-b\le 1$, then we have $\Im h_2(re^{\pi i})\ge 0$ because 
$\sin (a-b)\pi$, $\sin a \pi$, $-\sin b\pi \ge 0$.
Using the same method as the proof of Theorem 2.2, 
we can get the operator monotonicity of $h_2$.
If $0<a=-b\le 1$, then 
\[  \Im h_2(re^{\pi i}) = \frac{1}{|z^{-a}+1|^2}(2\cos a\pi +r^a+r^{-a})\sin a\pi \ge 0. \] 
So is $h_2$.

If $a=1$ and $-1<b<0$, then we have
\[  \Im h_2(re^{\pi i}) = \frac{1}{|z^b+1|^2}r^b(r-1) \sin b\pi <0  \]
for some $r>0$.
If $b=-1$ and $0<a<1$, then we have
\[  \Im h_2(re^{\pi i}) = \frac{1}{|z^{-1}+1|^2} r^a(1-\frac{1}{r})\sin a\pi <0  \]
for some $r>0$.
So, for these 2 cases, $h_2$ is not operator monotone.
\end{proof}
We consider the function $h_1$ (resp. $h_2$) as an extension of the representing
function $M_\alpha$ (resp. $L_p$) of the power difference mean (resp. the Lehmer mean),
where
\begin{gather*}
  M_\alpha (t) = \frac{\alpha -1}{\alpha} \frac{t^\alpha -1}{t^{\alpha -1} -1}, \qquad (-1\le \alpha \le 2)  \\
\intertext{and}
  L_p(t) =  \frac{t^p+1}{t^{p-1}+1}, \qquad (0 \le p \le 1) 
\end{gather*}
(see \cite{hiaikosaki}, \cite{nakamura}). 



%

\end{document}